\documentclass[11pt]{article}

\usepackage{amsthm}
\usepackage{amsmath}
\usepackage{amssymb}
\usepackage{color}
\usepackage[bookmarks]{hyperref}
\usepackage{paralist}
\usepackage[numbers]{natbib}
\usepackage{verbatim}
\usepackage{url}
\usepackage[all]{xy}

	\usepackage{soul}
	\definecolor{lightblue}{rgb}{.60,.60,1}

        \definecolor{brown}{rgb}{.75,.75,.75}
	
\title{Arithmetic complexity via effective names for random sequences}

\author{
	Bj{\o}rn Kjos-Hanssen\thanks{This material is based upon work supported by the National Science Foundation under Grants No.\ 0652669 and 0901020.} \\
		\textit{University of Hawai\textquoteleft i at M\=anoa} \\
		\url{bjoern@math.hawaii.edu} \and
	Frank Stephan\thanks{F. Stephan is partially supported by
NUS grant R252-000-420-112.} \\
		\textit{National University of Singapore} \\
		\url{fstephan@comp.nus.edu.sg} \and
	Jason Teutsch\thanks{J.~Teutsch is supported by
the Deutsche Forschungsgemeinschaft grant ME 1806/3-1.} \\
		\textit{Ruprecht-Karls-Universit\"at Heidelberg} \\
		\url{teutsch@math.uni-heidelberg.de}
}

\newcommand{\mc}{\mathcal}
\newcommand{\A}{\mc{A}}
\newcommand{\B}{\mc{B}}
\newcommand{\C}{\mc{C}}

\newcommand{\phe}{\varphi}

\newcommand{\0}{\emptyset}

\newcommand{\T}{\mathrm{T}}
\newcommand{\m}{\mathrm{m}}
\newcommand{\lex}{\mathrm{lex}}

\newcommand{\andd}{\quad \& \quad}

\newcommand{\join}{\mathrel{\oplus}}
\newcommand{\intersect}{\mathrel{\cap}}
\newcommand{\union}{\mathrel{\cup}}
\DeclareMathOperator{\set}{set}
\DeclareMathOperator{\real}{real}

\renewcommand{\complement}[1]{\overline{#1}}

\newcommand{\restr}{\upharpoonright}

\newcommand{\converge}{\mathop\downarrow}
\newcommand{\diverge}{\mathop\uparrow}
\DeclareMathOperator{\dom}{dom}

\newcommand{\grow}{_\lrcorner}

\newcommand{\lcode}{\left\langle}
\newcommand{\rcode}{\right\rangle}
\newcommand{\pair}[1]{{\lcode#1\rcode}}
\newcommand{\size}[1]{{\left|#1\right|}}

\theoremstyle{plain}        	\newtheorem{thm}{Theorem}[section]
\theoremstyle{definition} 	\newtheorem{defn}[thm]{Definition}
\theoremstyle{plain}        	\newtheorem{prop}[thm]{Proposition}
\theoremstyle{plain}        	
\theoremstyle{plain}        	\newtheorem{cor}[thm]{Corollary}
\theoremstyle{plain}        	\newtheorem{fact}[thm]{Fact}
\theoremstyle{plain}        	\newtheorem{lemma}[thm]{Lemma}
\theoremstyle{remark}  	
\theoremstyle{remark}  	\newtheorem{rem}[thm]{Remark}
\theoremstyle{plain}      	\newtheorem{ques}[thm]{Question}
\theoremstyle{plain}	\newtheorem*{sigma3repthm}{$\Sigma_3$-Representation Theorem}
\newcommand{\nlb}{\nolinebreak[3]}

\def\noqed{\renewcommand{\qedsymbol}{}}

\numberwithin{equation}{section}

\begin{document}
\maketitle

\begin{abstract}
We investigate enumerability properties for classes of sets which permit recursive, lexicographically increasing approximations, or \emph{left-r.e.} sets.  In addition to pinpointing the complexity of left-r.e.\ Martin-L\"{o}f, computably, Schnorr, and Kurtz random sets, weakly 1-generics and their complementary classes, we find that there exist characterizations of the third and fourth levels of the arithmetic hierarchy purely in terms of these notions.  More generally, there exists an equivalence between arithmetic complexity and existence of numberings for classes of left-r.e.\ sets with shift-persistent elements.  While some classes (such as Martin-L\"{o}f randoms and Kurtz non-randoms) have left-r.e.\ numberings, there is no canonical, or \emph{acceptable}, left-r.e.\ numbering for any class of left-r.e.\ randoms.  Finally, we note some fundamental differences between left-r.e.\ numberings for sets and reals.
\end{abstract}

\section{Effective randomness}

Think of a real number between $0$ and $1$.  Is it random?  In order to give a meaningful answer to this question, one must first obtain an expression for the real number in mind.  Any reasonable language contains no more than countably many expressions, and therefore we must always settle for a language with uncountably many indescribable reals.  On the other hand, there exists a natural and robust class of real numbers which admit recursive increasing approximations.  We call such numbers \emph{left-r.e.}\ reals.  Brodhead and Kjos-Hanssen \citep{BK09} observed that there exists an effective enumeration, or \emph{numbering}, of the left-r.e.\ reals, and Chaitin \citep{Cha87} showed that some left-r.e.\ reals are Martin-L\"of random.  Random left-r.e.\ reals thus serve as a friction point between definability and pure randomness.

In the following exposition we examine which classes of left-r.e.\ randoms and non-randoms admit numberings (and are therefore describable). A related definability question also arises, namely \emph{how {difficult} is it to determine whether a real is random?}  As a means of classifying complexity, we place the index sets for left-r.e.\ randoms inside the arithmetic hierarchy.  One can view this program as a continuation of work by Hitchcock, Lutz, and Terwijn \citep{HLT07} which places classes of randoms inside the broader Borel hierarchy.  In contrast with the case of r.e.\ sets, we shall find a close connection between numberings and arithmetic complexity for classes of left-r.e.\ reals.

\paragraph{Notation.} Some standard notation used in this article includes \emph{$\forall^\infty$} which denotes ``for all but finitely many'' and $\exists^\infty$ which means ``there exist infinitely many.''  $X \restr n$ is the length $n$ prefix of $X$, and \emph{$^\frown$} denotes concatenation.   For finite sequences $\sigma$ and $\tau$, $\sigma \preceq \tau $ means that $\sigma$ is a prefix of $\tau$, $\sigma \prec \tau$ indicates that $\sigma$ is a proper prefix of $\tau$, and $\size{\sigma}$ is the length of $\sigma$.  For non-negative integers $x$, $\size{x}$ is the floor of $\log (x+1)$.  $\pair{\cdot,\cdot}: \omega \times \omega \mapsto \omega$ is some recursive pairing function which we fix for rest of the paper.  For sets $A$ and $B$, $A \join B = \{2n : n \in A\} \union \{2n+1 : n \in B\}$.  $'$ is the jump operator,  $\mu$ is the unbounded search operator, $\converge$ denotes convergence, and $A \leq_\T B$ means $A$ Turing reduces to $B$. As usual, $\0'$ denotes the halting set, and $\complement{A}$ denotes the complement of the set $A$. For further background on recursion theory and algorithmic randomness, see \citep{Soa87} and \citep{DH10}.

A \emph{sequence} is the characteristic function of a set of natural numbers, and each sequence $A$ corresponds to a unique real number 
\[
	\real(A) = \sum_n 2^{-n-1} \cdot A(n).
\]  
We denote the class of all sequences by $\{0,1\}^\omega$, and $\{0,1\}^*$ is the class of finite strings.  A partial recursive function (synonymously, a \emph{machine}) $M$ is said to be \emph{prefix-free} if for any finite strings $\sigma, \tau \in \dom M$, $\sigma$ is not a proper prefix of $\tau$.  The \emph{prefix-free complexity} of a string $\sigma$ with respect to a prefix-free machine $M$ is given by $K_M(\sigma) = \min \{\size{p} : M(p) = \sigma\}$.  Furthermore, there exists a \emph{universal} prefix-free machine $U$ such that for any prefix-free machine $M$, $K_U(\sigma) \leq K_M(\sigma) + O(1)$ for all $\sigma \in \{0,1\}^*$ \citep{LV08}.  We fix such a $U$ and let $K = K_U$ for the remainder of this exposition.

\begin{defn}\label{mlr}
A sequence $X$ is called \emph{Martin-L\"{o}f random} \citep{Lev74, ML66} if
\begin{equation} \label{eqn: defn ML random}
	(\exists c)\: (\forall n)\: [K(X \restr n) \geq n - c].
\end{equation}
\end{defn}

\noindent
Intuitively, every prefix of the string $X$ in \eqref{eqn: defn ML random} is incompressible and therefore admits no simple description.

A \emph{martingale} $M:\{0,1\}^* \to \mathbb{R} \intersect [0,\infty)$ is a function satisfying the {fairness condition}: for all $\sigma\in 2^{<\omega}$,
\[
	M(\sigma) = \frac{M(\sigma 0) + M(\sigma 1)}{2}.
\]
The martingale $M$ \emph{succeeds} on a sequence $X$ if $\lim\sup M(X \restr n) = \infty$.  If  $M$ succeeds on $X$ and there exists a recursive, non-decreasing, unbounded function $g$ satisfying $g(n) \leq M(X \restr n)$ for infinitely many $n$, we say that $M$ \emph{Schnorr-succeeds} on $X$.  A martingale $M$ \emph{Kurtz-succeeds} on a set $A$ if $M$ succeeds on $A$ and there exists a recursive, non-decreasing, unbounded function $f$ such that $M(A \restr n) > f(n)$ for all $n$.  The idea behind Definition~\ref{cr} is that no gambling strategy can achieves arbitrary wealth by betting on a random sequence.

\begin{defn}\label{cr}
A sequence $X$ is called \emph{computably random} \citep{Sch71a,Sch71b} if no recursive martingale succeeds on $X$, \emph{Schnorr random} \citep{Sch71b} if no recursive martingale Schnorr-succeeds on $X$, and \emph{Kurtz random} \cite{DGR04, Kur81, Wan96} if no recursive martingale Kurtz-succeeds  on $X$.
\end{defn}

\noindent
The classes of randoms mentioned above relate to each other as follows:
\begin{thm}[see \citep{DH10} or \citep{N09}] \label{ML implies computable implies Schnorr}
Martin-L\"{o}f randomness $\implies$ computable randomness $\implies$ Schnorr randomness $\implies$ Kurtz randomness.
\end{thm}

Our discussion will also involve a related class of sequences which we introduce in Definition~\ref{def:1-generic}.  A set of finite strings $S$ is called \emph{dense} if for every string $\sigma$ there exists $\tau \in S$ extending $\sigma$.
\begin{defn} \label{def:1-generic}
A sequence is \emph{weakly 1-generic} if it has a prefix in every dense r.e.\ set of strings.   Even stronger, a sequence $X$ is \emph{1-generic} if for every r.e.\ set of strings~$A$, either~$X$ has a prefix in~$A$ or there exists a prefix of~$X$ which has no extension in~$A$.  
\end{defn}
\noindent
While a left-r.e.\ real cannot be 1-generic \cite[Proposition~XI.2.3]{Odi99}, weakly 1-generic sets can be left-r.e. \cite{JST11, N09, ST10b}.  We shall make use of the following result of Kurtz which also appears in \cite[Theorem~8.11.7]{DH10}.
\begin{thm}[\citet{Kur81}] \label{thm: 1-generic implies Kurtz random}
Every weakly 1-generic is Kurtz random.
\end{thm}

\section{Sets, reals, and acceptable numberings} \label{sec: acceptable}
We turn our attention to the magical correspondence between binary sequences and reals in $[0,1]$.  In particular, each binary sequence is the binary expansion of some real number and vice-versa.  We call a real \emph{non-dyadic} if its binary expansion contains both infinitely many 1's and infinitely many 0's, and \emph{dyadic} otherwise.  This definition highlights an important distinction between sets and reals.  For any string $\sigma$, the real number $.\sigma 011111\dotsc$ equals $.\sigma100000\dotsc$.  Hence there is no difference between the set of ``finite'' reals and the set of ``co-finite'' reals.  For the same reason, and unlike the case for sequences, there is no difference between ``infinite'' and ``co-infinite'' reals.  We shall use $\real({A})$ to denote the unique real representation of a set $A$ and $\set({X})$ to denote an arbitrarily selected set representation of a real $X$.

In general, enumerability will depend on whether we view our objects of study as sequences or as reals, see Remark~\ref{rem: sets are not reals}.  Indeed, sequence enumerations are more restrictive than real enumerations.  For random objects, however, the choice of sequences versus reals is immaterial since random reals are non-dyadic.  Every random real corresponds to a unique random sequence (which in turn corresponds uniquely to the characteristic function of a set) and vice-versa.  The identification of finite and co-finite sets leads to ambiguity in terminology and reference, hence we favor sets over reals throughout this exposition.  Nevertheless, we keep in mind the correspondence between sets and reals and occasionally exploit their relationship.  Where the discussion does not benefit from distinction between random reals, random sets, or random sequences, we may simply refer to objects as \emph{randoms}.

A set $A$ is called \emph{left-r.e.}\footnote{Our definition is analogous to the usual definition of \emph{left-r.e.}\ for reals which requires that the real admits a recursive approximation from below.  In more detail, a real number $X \in [0,1]$ is called \emph{left-r.e.}\ if it can be written in the form
\[
X = \sum_{x \in \dom \phe} 2^{-\size{x}}.
\]
for some numbering $\phe$.}\ if there exists a uniformly recursive approximation $A_0, A_1, A_2, \dotsc$ to $A$ such that $A_s \leq_\lex A_{s+1}$ for all $s$.  Here \emph{$A_s \leq_\lex A_{s+1}$} means that either $A_{s+1} = A_s$ or the least element $x$ of the symmetric difference satisfies $x \in A_{s+1}$. Left-r.e.\ sets are limit-recursive sets with recursive approximations of a special form.  We call $A_0, A_1, A_2, \dotsc$ a \emph{left-r.e.\ approximation of $A$}.   Every r.e.\ set is left-r.e.\ as $A_s \subseteq A_{s+1}$ implies $A_s \leq_\lex A_{s+1}$. Zvonkin and Levin \cite{ZL70} and later Chaitin \citep{Cha87} showed that there exists a left-r.e.\ Martin-L\"{o}f random set. (Like Chaitin we will fix one and call it $\Omega$.) It follows that each of the classes in Theorem~\ref{ML implies computable implies Schnorr} contains a left-r.e.\ member.

A \emph{numbering} $\phe$ is a partial-recursive (p.r.) function $\pair{e,x} \mapsto \phe_e(x)$.  A numbering $\phe$ is a programming language, and $\phe_e$ is the $e^\text{th}$ program in that language.  While $\phe$ enumerates p.r.\ functions, our main focus in this paper will be enumerations of sets and reals which admit recursive approximations from below.
\begin{defn}
Let $\C$ be a class of left-r.e.\ sets\footnote{\label{foot: real numbering} For reals, the definition of left-r.e.\ numbering would be similar but, as we see from Remark~\ref{rem: sets are not reals}, not equivalent.  A \emph{left-r.e.\ numbering} of a class of left-r.e.\ reals $\C$ is a function with range $\C$ given by
\begin{equation} \label{def:left-r.e. numbering}
e \mapsto \sum_{\sigma \in \dom \phe_e} 2^{-\size{\sigma}}
\end{equation}
for some numbering $\phe$.}.  A \emph{left-r.e.\ numbering} $\alpha$ of $\C$ is a p.r.\ function from natural numbers to $\C$ given by
\[
e \mapsto \lim_{s \to \infty} \alpha_{e,s} = \alpha_e
\]
where:
\begin{enumerate}[\scshape (i)]
\item $\alpha_{e,s}$ is uniformly recursive in $e$ and $s$, and
\item $\alpha_{e,0}, \alpha_{e,1}, \alpha_{e,2}, \dotsc$ is a left-r.e.\ recursive approximation of $\alpha_e$.
\end{enumerate}
\end{defn}

The following definition is a terse review of the arithmetic hierarchy.  For a more in-depth discussion see \cite{Soa87}.  A set $A\subseteq\omega$ is a called a \emph{$\Sigma_n$ set} if it is $\Sigma^{0}_{n}$ in the usual sense of recursion theory. The complement of a $\Sigma_n$ set is a $\Pi_n$ set.  We say that a set $A$ \emph{many-to-one} reduces to a set $B$, or \emph{$A \leq_\m B$}, if there exists a recursive function $f$ such that for all $x$, $x \in A \iff f(x) \in B$.  A set $A$ is called \emph{$\Sigma_n$-hard} (resp.\ \emph{$\Pi_n$-hard}) if for every $\Sigma_n$ (resp.\ $\Pi_n$) set $X$, $X \leq_\m A$.  A set $A$ is \emph{$\Sigma_n$} (resp.\ \emph{$\Pi_n$}) \emph{complete} if $A$ is a $\Sigma_n$ (resp.\ $\Pi_n$) set and $A$ is $\Sigma_n$-hard (resp. $\Pi_n$-hard).  The \emph{index set} for a class $\C$ with respect to a (left-r.e.) numbering $\alpha$ is $\{e : \alpha_e \in \C\}$.

We make use of the following classical theorem, and we will prove an analogue for left-r.e.\ index sets in Theorem~\ref{Sigma_3 iff numbering}.

\begin{sigma3repthm}[\citep{Soa87}] \label{thm: Sigma_3 representation}
Let $W_0, W_1, W_2, \dotsc$ be an acceptable universal r.e.\ numbering, and let $A$ be a $\Sigma_3$-set.  Then there exists a recursive function $f$ such that for all $x$,
\begin{align*}
x \in A &\implies (\forall^\infty y)\: [W_{f(x,y)} = \omega]; \\
x \notin A &\implies (\forall y)\: [\text{$W_{f(x,y)}$ is finite}].
\end{align*}
\end{sigma3repthm}
\noindent

A left-r.e.\ numbering of all left-r.e.\ sets is called \emph{universal}.  Similarly, an \emph{r.e.\ numbering} of a class or r.e.\ sets is a mapping $e \mapsto \dom \phe_e$ for some numbering $\phe$, and an r.e.\ numbering is \emph{universal} if every r.e.\ set appears in its range.  Universal r.e.\ numberings are known to exist, see \cite[Definition~4.1]{Soa87}.  Universal left-r.e.\ numberings also exist \citep{BK09}: if $\phe_e$ induces a universal r.e.\ numbering, then $\phe_e$ induces a universal left-r.e.\ numbering.  

We shall use capital letters to denote sequences and sets, but we reserve the capital letter $W$ for r.e.\ numberings. Greek letters $\sigma$ and $\tau$ will denote finite binary strings, $\phe$ and $\psi$ will denote numberings, and $\alpha$, $\beta$, $\gamma$, and $\zeta$ will be left-r.e.\ numberings (with an exception in Theorem~\ref{thm: a set}).

The following result illustrates a crucial difference between left-r.e.\ reals and left-r.e.\ sets:
\begin{prop} \label{prop: bummer}
The co-infinite left-r.e.\ sets do not have a left-r.e.\ numbering.
\end{prop}
\begin{proof}
Suppose that such a numbering $\alpha$ exists, let $W_0, W_1, W_2, \dotsc$ be a universal r.e.\ numbering with $W_{d,0}, W_{d,1}, W_{d,2}, \dotsc$ a recursive approximation of $W_d$.  Then $W_d$ is co-infinite if and only if $W_d = \alpha_e$ for some $e$, that is:
\[
(\exists e)\: (\forall s,x)\: (\exists t>s)\: [\alpha_{e,t}(x) = W_{d,t}(x)].
\]
Thus $\{d : \text{$W_d$ is co-infinite}\}$ is $\Sigma_3$, contradicting the fact that this set is also $\Pi_3$-complete \cite[Corollary~3.5]{Soa87}.
\end{proof}
\begin{rem} \label{rem: sets are not reals}
On the other hand, every real belongs to the equivalence class of some co-infinite set because every dyadic rational can be represented using infinitely many 0's and finitely many 1's, and every non-dyadic rational can be represented using infinitely many 0's and infinitely many 1's. Since there exists a left-r.e.\ numbering for the class of left-r.e.\ reals \cite{BK09}, the co-infinite left-r.e.\ reals have a left-r.e.\ numbering in contrast to the corresponding result for sets (Proposition~\ref{prop: bummer}).
\end{rem}

Theorem~\ref{thm: a set} more precisely describe the relationship between enumerations of left-r.e.\ sets and left-r.e.\ reals.  A [left-r.e.\ or r.e.]~numbering is called a [left-r.e.\ or r.e.]~\emph{one-one} numbering or left-r.e.\ \emph{Friedberg} numbering if every member in its range has a unique index.
\begin{thm} \label{thm: a set}
A set $\C$ of nonzero reals between $0$ and $1$ has a left-r.e.\ numbering~$\alpha$ (in the sense of Footnote~\ref{foot: real numbering}) iff the class of sets
\[
\{A\colon \text{$A$ is infinite and $\real(A) \in \C$}\}
\]
has a left-r.e.\ numbering.  The same holds for left-r.e.\  one-one numberings.
\end{thm}
\begin{proof}
$\Longrightarrow$: Let $\alpha_0,\alpha_1,\ldots$ be a (one-one) enumeration
with dyadic approximations $\alpha_{e,s}$ to $\alpha$, let
\[
A_{e,s} = \set[(1-3^{-s}) \cdot \alpha_{e,s}],
\]
and let $A_e = \lim_s A_{e,s}.$  Since $\real(A_e) = \alpha_e$ for all $e$, it remains only to show that $A_e$ is infinite.  If $\alpha_e$ is non-dyadic, then $A_e$ is the unique infinite set with $\real(A_e) = \alpha_e$.  Otherwise $\alpha_e$ is dyadic, in which case all the sets $A_{e,s}$ are lexicographically less than $A_e$ and so $A_e$ is co-finite.  Finally, the numbering $A$ is one-one whenever the numbering $\alpha$ is.

$\Longleftarrow$: If $A_0, A_1, A_2, \dotsc$ is a list of infinite r.e.\ sets then the reals
\[
\alpha_{e,s} = \sum_{\{x<s \colon x \in A_{e,s}\}} 2^{-x-1} \cdot A_{e,s}(x)
\]
approximate uniformly in $e$ the numbers $\real(A_e)$ from below. Again
if the numbering $A$ is one-one then so is $\alpha$.
\end{proof}

Garden variety numberings in recursion theory satisfy the $s$-$m$-$n$~Theorem \citep{Soa87} and are called \emph{acceptable} numberings:
\begin{defn} \label{defn: acceptable numberings}
A (left-r.e.)~numbering $\phe$ is called a (left-r.e.) \emph{G\"{o}del numbering} or \emph{acceptable (left-r.e.)~numbering} if for every (left-r.e)~numbering $\psi$ there exists a recursive function $f$ such that $\phe_{f(e)} = \psi_e$ for all $e$.
\end{defn}

\noindent
Intuitively, the function $f$ in Definition~\ref{defn: acceptable numberings} translates code from program $\psi$ into program $\phe$.  Thus acceptable numberings are maximal: any given numbering can be uniformly translated into any acceptable one.  Furthermore, any two acceptable numberings are isomorphic in the sense of \citep{Rog58}.  These two properties make the notion of an acceptable numbering rather robust.  Moreover, the existence of an acceptable numbering is in a sense equivalent to Church's Thesis via the $s$-$m$-$n$~Theorem~\cite{Soa87}.

We show that there is no canonical way to number random sets via acceptable left-r.e.\ numberings.  The class of left-r.e.\ random reals is a natural example of a class which has a left-r.e.\ numbering but no maximal (i.e.\ acceptable) numbering.

\begin{defn}
Let $\C \subseteq \{0,1\}^\omega$.  A set $X$ is a \emph{shift-persistent element} of $\C$ if $\sigma{^\frown}X \in \C$ for every string $\sigma$.
\end{defn}

\begin{thm}
Assume that a family $\C$ has a shift-persistent element and there exists an infinite left-r.e.\ set $R <_\lex \omega$ with $R\not\in\C$.  Then $\C$ does not have an acceptable left-r.e.\ numbering.
\end{thm}

\begin{proof}
Let $X$ be a shift-persistent member of $\C$, let $R$ be the missed out infinite set with $R <_\lex \omega$, and let $\sigma_0,\sigma_1,\sigma_2,\ldots$ be a left-r.e.\ approximation of $R$ such that all $n$ satisfy $\sigma_n 1 1 1 1 \ldots <_\lex \sigma_{n+1} 0 0 0 0 \ldots <_\lex R$.  Every infinite left-r.e.\ set has such an approximation.  Suppose $\alpha$ is an acceptable left-r.e.\ numbering of $\C$.

Fix a left-r.e.\ approximation $\Omega_0, \Omega_1, \Omega_2, \dotsc$ for $\Omega$, and let $c_{\Omega}(n)$ be the first stage for which this approximation has settled on the first $n$ positions. Note that $c_{\Omega}$ dominates every recursive function, otherwise we would infinitely often have $K(\Omega \restr n) \leq \log n + k$ for some constant $k$.  Now there is a $\0'$-recursive function $F$ such that $F(n)$ is the first $m$ such that the first $m$ bits of $R$ differ from the first $m$ bits of every $\alpha_k$ with $k \leq c_{\Omega}(n)$.  This function $F$ has an approximation $F_s$ and now one takes the set $\beta_n = \sigma_s {^\frown} X$ for the first stage~$s$ such that for all $t \geq s$ it holds that $F_t(n) = F_s(n)$ and the first $F_s(n)$ bits of $\sigma_t$ exist and are equal to those
of $\sigma_s$. Note that this $\sigma_s$ can be found as the function values $F_t(n)$ converge to $F(n)$ and similarly the $\sigma_t$ converge to $R$.

Each set $\beta_n$ is in the list $\alpha_0,\alpha_1,\alpha_2,\ldots$ by definition of $X$. Furthermore, $\beta_n$ coincides with $R$ on its first $F(n)$ bits while every $\alpha_k$ with $k \leq c_{\Omega}(n)$ differs from $R$ on its first $F(n)$ bits. Hence $\beta_n \notin \{\alpha_0,\alpha_1,\ldots,\alpha_{c_{\Omega}(n)}\}$. It follows that there is no recursive function $f$ with $\beta_n = \alpha_{f(n)}$ for all $n$ as $c_{\Omega}$ would dominate $f$. Thus the numbering $\alpha$ cannot be an acceptable numbering of the left-r.e.\ sets of its type.
\end{proof}

\noindent
It follows that there is no canonical way to enumerate random reals:
\begin{cor}
There is no acceptable left-r.e.\ numbering of either the left-r.e.\ randoms or the left-r.e.\ non-randoms (under any reasonable definition of random).
\end{cor}

\section{Arithmetic classification via numberings} \label{sec: acvn}

Unlike r.e.\ numberings, the existence of left-r.e.\ numberings admits a neat characterization in terms of $\Sigma_3$ sets.  As a corollary, we will get that the left-r.e.\ Martin-L\"{o}f random reals are enumerable but not co-enumerable.  In order to make concatenation easier, we introduce the following operator on finite strings.
\begin{defn}
For any finite binary string $\sigma$, $\sigma\grow$ denotes the string $\sigma$ with the maximum 1 changed to a 0 (if it exists).  If $\sigma$ consists of all zeros, then $\sigma\grow = \sigma$.
\end{defn}

A refinement of the following result appears in \cite[Theorem~3.5.21]{N09} using an alternate proof.

\begin{lemma}[\citet{N09}] \label{lots of zeros implies nonrandom}
Let $X$ be a sequence which infinitely often has a prefix of length $n$ followed by $n \cdot 2^n$ zeros.  Then $X$ is not Schnorr random.
\end{lemma}

\begin{proof}
We exhibit a martingale which Schnorr-succeeds on $X$.  The betting strategy is as follows.  For simplicity, let us assume that we start with \$3.  For the initial bet, place \$1 on the ``1'' outcome.  Now suppose we have already seen a string $\sigma$ of length $n$.  If the last digit of $\sigma$ is ``0,'' then bet $2^{-n}$ dollars on the ``1'' outcome.  Otherwise, make the same bet that was made the last time.

We claim this martingale succeeds on $X$.  The martingale loses at most $2^{-n}$ dollars from betting on the $(n+1)^\text{st}$ digit of $X$.  Thus the total money lost from playing over an infinite amount of time is at most \$2.  On the other hand, we are bound to eventually reach a string of consecutive zeros of length $n\cdot 2^n$ immediately following $X \restr n$.  At this point, $2^{-n}$ dollars will be wagered $n\cdot 2^n$ times in a row, for a net gain of $\$n$ over the interval of zeros.  By assumption on $X$ we reach such points infinitely often, and therefore the winnings go to infinity.

Finally we exhibit a recursive function which infinitely often is a correct lower bound for the gambler's capital.  Define a recursive function which guesses at each position that we are at the end of an interval of $n\cdot 2^n$ zeros.  The function always outputs $n$ where $n$ is the length of the corresponding interval that would have preceded the long string of zeros.  if no such integer $n$ exists, then output 0.  Infinitely often this guess will be correct and, as noted in the previous paragraph, we will indeed have at least $n$ dollars at this point.
\end{proof}
\noindent
Since weakly 1-generic sets are Kurtz random (Theorem~\ref{thm: 1-generic implies Kurtz random}), Proposition~\ref{prop: 1-generic zeros} below implies that Lemma~\ref{lots of zeros implies nonrandom} does not carry over for Kurtz random sequences.
\begin{prop}\label{prop: 1-generic zeros}
Let $X$ be weakly 1-generic sequence and let $f$ be a recursive function.  Then for infinitely many $n$, $(X\restr n)^{\frown} 0^{f(n)}$ is a prefix of~$X$.
\end{prop}
\begin{proof}
Let
\[
A_n = \{\sigma^\frown 0^{f(\size{\sigma})} \colon \size{\sigma} \geq n\}.
\]
For all $n$, some member of $A_n$ is a prefix of $X$ since $A_n$ is a dense r.e.\ set.  Suppose there are only finitely many prefixes of $X$ of the form $(X \restr n)^{\frown}0^{f(n)}$, and let $k$ be greater than the length of the longest such prefix.  Then some member of $A_k$ must also be a prefix of $X$, contradicting the definition of $k$.
\end{proof}

\begin{defn}
Let $\sigma_0, \sigma_1, \sigma_2, \dotsc$ be a sequence of strings where $\sigma_{e,s}$ is a stage~$s$ approximation of $\sigma_e$.  We will say that $\sigma_e$ \emph{blows up to infinity} if $\lim_s \size{\sigma_{e,s}} = \infty$, and $\sigma_e$ \emph{gets kicked to infinity} if $\sigma_j$ blows up to infinity for some $j < e$.
\end{defn}
\begin{thm} \label{thm: ML Sigma_3-hard}
Let $A \subseteq \omega$ be a $\Sigma_3$-set, and let $\alpha$ be an acceptable universal left-r.e.\ numbering.  Then there exist a recursive function $g$ such that
\begin{align*}
x \in A &\implies \text{$\alpha_{g(x)}$ is Martin-L\"{o}f random;} \\
x \notin A &\implies \text{$\alpha_{g(x)}$ is not Schnorr random.}
\end{align*}
\end{thm}
\begin{proof}
Let $W$ be an acceptable universal r.e.\ numbering.  Without loss of generality, assume that for all $e$ at most one element of $e$ enters $W_e$ at each stage of its enumeration $\{W_{e,s}\}$ and furthermore at least one $W_e$ increases at each stage.  By the $\Sigma_3$-Representation~Theorem, there exists a function $f$ satisfying:
\begin{align*}
x \in A &\implies \text{$W_{f(x,n)}$ is infinite for some $n$;} \\
x \notin A &\implies \text{$W_{f(x,n)}$ is finite for all $n$.}
\end{align*}
For each $x$ and $s$, let
\[
	\sigma^x_{0,s} = \Omega_s \restr \size{W_{f(x,0),s}},
\]
let
\begin{multline*}
m(e,s) = \text{greatest stage $t + 1 < s$ such that} \\ \max \{x: \Omega_{e,t+1}(x) =1\} \neq \max \{x: \Omega_{e,t}(x) = 1\},
\end{multline*}
and inductively define
\begin{equation}
\sigma^x_{n+1,s}  = 1^{(|\sigma^x_{n,s}| + 2) \cdot 2^{|\sigma^x_{n,s}|}} {^\frown} \left(\Omega_s  \restr \size{W_{f(x,n+1),m[f(x,n+1),s]}}\right)\grow.
\end{equation}
Roughly speaking, $\sigma^x_{n+1,s}$ consists of a long string of 1's followed by an approximation of $\Omega$.  Define the recursive function $g$ by
\begin{equation} \label{eqn2: ML Sigma_3-hard}
\alpha_{g(x)} = \lim_s \sigma^x_{0,s} {^\frown} \sigma^x_{1,s} {^\frown} \sigma^x_{2,s} \dotsc
\end{equation}
By Lemma~\ref{lots of zeros implies nonrandom}, there are enough 1's that if all the $\sigma^x_n$'s remain finite, then \eqref{eqn2: ML Sigma_3-hard} is not Schnorr random.  On the other hand, if some $\sigma^x_n$ does blow up to infinity, then \eqref{eqn2: ML Sigma_3-hard} becomes the Martin-L\"{o}f random $\Omega$ with some finite prefix attached.

We verify that the approximation in \eqref{eqn2: ML Sigma_3-hard} is left-r.e.\ by analyzing the change between stages $s$ and $s+1$.   By induction, the length of $\sigma^x_{n,t}$ is increasing in $t$ for every $n$.  Let $e$ be the least index such that $\sigma^x_{e,s+1}$ is longer than $\sigma^x_{e,s}$.  By minimality, the prefix of 1's at the beginning of this string must remain unchanged but the approximation to $\Omega$ increases.   In particular,
\[
	\size{W_{f(x,e),m[f(x,e),s]}} \neq \size{W_{f(x,e),m[f(x,e),s+1]}}.
\]
Due to the $\grow$ operator, the 0 at some existing position changes to a 1 in stage $s+1$.  Hence $\sigma^x_e$ can expand in stage $s+1$ while permitting a left-r.e.\ approximation for $\eqref{eqn2: ML Sigma_3-hard}$.  Finally, the limit in $\eqref{eqn2: ML Sigma_3-hard}$ exists because the sequence of reals is increasing and bounded from above.

Suppose that $W_{f(x,n)}$ is infinite for some $n$, and let $e$ be the least such index.  By minimality, $\sigma^x_j = \lim_s \sigma^x_{j,s}$ is finite for all $j < e$.  Hence for $e > 0$,
\[
	\alpha_{g(x)} = \sigma^x_0 {^\frown} \sigma^x_1 {^\frown} \sigma^x_2 {^\frown} \dotsb {^\frown} 1^{(|\sigma^x_e| + 2) \cdot 2^{|\sigma^x_e|} } {^\frown} \Omega,
\]
which is Martin-L\"{o}f random.  All $\sigma^x_n$ with $n >e$ gets kicked to infinity.  The case $e=0$ is similar.

On the other hand, suppose that $W_{f(x,n)}$ is finite for all $n$.  In this case $\sigma^x_0$ is finite, and
\[
	\sigma^x_{n+1} = 1^{(|\sigma^x_n| + 2) \cdot 2^{|\sigma^x_n|} } {^\frown}   \left(\Omega_{s_{n}} \restr \size{W_{f(x,n+1),m[f(x,n+1),s_{n}]}}\right)\grow,
\]
where $s_n$ is the final stage where $W_{f(x,n+1)}$ increases.  Thus infinitely often $\alpha_{g(x)}$ has a prefix of length $\size{\sigma}$ followed by $(\size{\sigma}+2) \cdot 2^{\size{\sigma}}$ 1's.  By Lemma~\ref{lots of zeros implies nonrandom}, $\alpha_{g(x)}$ is not Schnorr random.
\end{proof}

\begin{cor} \label{cor: ML randoms are Sigma_3-hard}
In any acceptable universal left-r.e.\ numbering, the indices of the left-r.e.\ Martin-L\"{o}f randoms are $\Sigma_3$-hard.
\end{cor}

Recall that a left-r.e.~numbering is called a left-r.e.~\emph{Friedberg} numbering if every member in its range has a unique index.  Friedberg initiated the study of these numberings in 1958 when he showed that the r.e.\ sets can be enumerated without repetition \citep{Fri58}.  More recently Kummer \citep{Kum90} gave a simplified proof of Friedberg's result, and Brodhead and Kjos-Hanssen \citep{BK09} adapted his idea to show that there exists a left-r.e.\ Friedberg numbering of the left-r.e.\ Martin-L\"{o}f random sets.  We now show that left-r.e.\ Friedberg numberings can be used to characterize $\Sigma_3$-index sets.

\begin{thm} \label{Sigma_3 iff numbering}
Let $\C$ be a class of infinite left-r.e.\ reals which contains a shift-persistent element.  Then for any universal left-r.e.\ numbering $\alpha$, the following are equivalent:
\begin{enumerate}[\scshape (i)]
\item $\{e : \alpha_e \in \C\}$ is a $\Sigma_3$-set.
\item There exists a left-r.e.\ numbering of $\C$.
\item There exists a left-r.e.\ Friedberg numbering of $\C$.
\end{enumerate}
\end{thm}
\begin{proof}
Let $\alpha$ be any universal left-r.e.\ numbering, and let
\[
	\C_\alpha = \{e: \alpha_e \in \C \}.
\]
\begin{proof}[\textsc{(i)}$\iff$\textsc{(ii)}]
Suppose that $\beta$ is a left-r.e.\ numbering for $\C$.  Then
\begin{align*}
\C_\alpha &= \{e : (\exists d)\: [\alpha_e = \beta_d] \} \\
&= \left\{e: (\exists d)\: (\forall n,s)\: (\exists t>s)\: [\alpha_{e,t} \restr n = \beta_{d,t} \restr n] \right\},
\end{align*}
so $\C_\alpha$ is a $\Sigma_3$ set.

Conversely, assume that $\C_\alpha \in \Sigma_3$ and let $\gamma$ be an acceptable universal left-r.e.\ numbering.  By Theorem~\ref{thm: ML Sigma_3-hard}, there exists a recursive function $g$ such that
\[
	e\in\C_\alpha \iff \text{$\gamma_{g(e)}$ is Martin-L\"{o}f random}.
\]
For sets $X$, let
\[
	r_b(X) = \max \{n : (\forall m\le n)\: [K(X \restr m) \geq m - b]\},
\]
and in case $X$ has a recursive approximation $X_0, X_1, X_2, \dotsc$, then we define a monotonic approximation to $r_b$ as follows:
\[
	r_{b,s+1}(X) = \max \{r_{b,s}(X), \max \{n : (\forall m\le n)\: [K_{s}(X_{s}\restr m)\geq m - b]\}\},
\]
where $K_{s}$ is a monotonically decreasing computable approximation to $K$. It may not be the case that $\lim_s r_{b,s}(X) = r_b(X)$, however we do achieve $\lim_s r_{b,s}(X) = \infty \iff r_b(X) = \infty$.

Without loss of generality, assume that $\alpha_{e,s}$ has finitely many 1's at each stage $s$ of the recursive approximation.  Let $C$ be a shift-persistent element of $\C$, and let $C_0, C_1, C_2 \dotsc$ be a left-r.e.\ approximation for $C$.  Since we want to avoid dealing with $\alpha_e$'s which are equal to 0, let
\[
	f(e) = \text{$e^\text{th}$ $\alpha$-index found to be nonzero},
\]
and let $t(e)$ be the first stage at which $\alpha_{f(e)}$ appears to be nonzero.  For notational convenience, let
\[
	q(e) = \min \{x : \alpha_{f(e),t(e)}(x)=1\},
\]
and let
\[
	\xi_{\pair{e,b},s} =
	\begin{cases}
	0^{q(e)} & \text{if $\size{r_{b,s}(\gamma_{g[f(e)],s})} \leq q(e)$;} \\
	\alpha_{f(e),s} \restr r_{b,s}(\gamma_{g[f(e)]}) &\text{otherwise}\\
	\end{cases}
\]
be the prefix of $\alpha_{f(e),s}$ that has the length of $\gamma_{g[f(e)]}$'s prefix which looks random at stage~$s$. Let
\begin{multline*}
m(e,s) = \text{greatest stage $t + 1 < s$ such that} \\ \max \{x: \alpha_{f(e),t+1}(x) =1\} \neq \max \{x: \alpha_{f(e),t}(x) = 1\}.
\end{multline*}
Define a further left-r.e.\ numbering $\beta$ by
\begin{equation} \label{eqn: Sigma_3 beta}
\beta_{\pair{e,b},s+1} = \xi_{\pair{e,b},m(e,s){\grow}}{^\frown}C_{s+1}.
\end{equation}
The operator $\grow$ in \eqref{eqn: Sigma_3 beta} is needed to ensure that $\beta$ is a left-r.e.\ numbering: whenever $r_{b,m(e,s+1)}(\gamma_{g[f(e)]}) \neq r_{b, m(e,s)}(\gamma_{g[f(e)]})$, this expansion is handled by replacing a ``0'' with ``1'' which clears the higher indices, making room for $C_{s+1}$.

Finally, $\beta_0, \beta_1, \dotsc$ is a left-r.e.\ numbering for $\C$.  Indeed,
\begin{eqnarray*}
f(e) \in \C_\alpha &\implies & \text{$(\exists b)\: [\gamma_{g[f(e)]}$ is Martin-L\"{o}f random with constant $b$]} \\
 &\implies & \beta_\pair{e,b} = \alpha_{f(e)}.
\end{eqnarray*}
Of course a $\beta$-index for the real 0 can be added if necessary.  In the case where $\gamma_{g[f(e)]}$ is not Martin-L\"{o}f random with constant $b$, $C_s$ does not get kicked to infinity but then $\beta_\pair{e,b} \in \C$ because $C$ is a shift-persistent member of $\C$.
\noqed
\end{proof}

\begin{proof}[\textsc{(ii)}$\iff$\textsc{(iii)}]
Assume that $\C$ has a numbering $\gamma$.  Let $C$ be a shift-persistent element of $\C$, and let
\[
	\B = \{1^n {^\frown} C : n \in \omega\} \union \{X \in \C : X \leq C\}
\]
be a subclass of $\C$.  $\B$ is the union of two classes which have left-r.e.\ numberings and therefore has itself a left-r.e.\ numbering.  A numbering for the latter class is achieved by pausing the enumeration of $X$ whenever it tries to exceed $C$.  Let $\beta$ be a left-r.e.\ numbering for $\B$.  

Note that
\[
	\A := \{X : X \in \C  - \B\} =  \{X \in \C : (\exists n)\: [1^n {^\frown} C < X < 1^{n+1} {^\frown} C] \}
\]
has a left-r.e.\ numbering $\alpha$ given by: $	\alpha_{\pair{e,n,k},s}=$
\[
	\begin{cases}
	1^n {^\frown} C_s + 2^{-k} & \text{if $(\gamma_{e,s} \restr k){^\frown}0  \leq_\lex (1^n {^\frown} C_s \restr k){^\frown}0$;} \\
	\gamma_{e,s} & \text{if $(1^n {^\frown} C_s \restr k){^\frown}0 <_\lex (\gamma_{e,s} \restr k){^\frown}0 <_\lex (1^{n+1} {^\frown} C_s \restr k){^\frown}0$;} \\
	1^{n+1} {^\frown} C_s - 2^{-k} & \text{if $(1^{n+1}{^\frown}C_s \restr k){^\frown}0 \leq_\lex (\gamma_{e,s} \restr k){^\frown}0$.}
	\end{cases}
\]
where the triple $\pair{e,n,k}$ ranges over values $k$ which are greater than or equal to the index of the least 0 in $1^n{^\frown} C$.  The numbering $\alpha$ exploits the fact that if $X \neq 1^n{^\frown}C$, then $X$ and $1^n{^\frown}C$ must differ on some prefix.  Strictly speaking, every \emph{tail} of $C$ must be a shift-persistent element in order that each $\alpha$-index yields a member of $\C$.  Since every member of $\C$ is infinite, however, we can overcome this shortcoming by modifying the tails for $\alpha_{\pair{e,n,k},s}$ to be $C_s$ in the first and third cases.

Using $\alpha$ and $\beta$, we now exhibit a Friedberg numbering $\zeta$ for $\A \union \B = \C$.  Let
\[
 M = \{ e: (\forall j < e)\: [\alpha_j \neq \alpha_e]\}.
\]
Every member of $\A$ has a unique index in $M$.  Since $M$ is a $\Sigma_2$-set, there exists a $\0'$-recursive function $m$ whose domain is $M$.  Let $m_0, m_1, m_2, \dotsc$ be a recursive approximation to $m$.  Using this approximation, we shall design $\zeta$ in such a way that each $\alpha$-indexed real in $M$ occurs at exactly one $\zeta$-index, and the remaining $\zeta$-indices will be home to the $\beta$-indexed reals.

We define a function $f: \omega \mapsto (\omega \union \{\infty\}) \times \{\alpha, *\}$ which maps $\zeta$-indices to either $\alpha$-indices or *'s.  The $\infty$ symbol is used for destroyed $\beta$ indices which are (or never were) attached to $\zeta$-indices, and the $\alpha$ and $*$ symbols indicate whether the particular $\zeta$-index is following an $\alpha$-index or a $\beta$-index.  If $f(e) = \pair{x,*}$ for some $x$, $f(e)$ ``explodes'' and we say that the $\zeta$-index $e$ has been \emph{destroyed}.  $f_s: \omega \mapsto ((\omega \union \{\infty\}) \times \{\alpha, *\}) \union \{\diverge\}$ will be a recursive approximation to $f$ based on the recursive approximation $m_s$.  $\zeta$-indices that are destroyed at some stage take on $\beta$-indices in the limit (rather than $\alpha$-indices).  We shall also keep track of which $\beta$-indices have been taken on by $\zeta$-indices: $G_s$ will be the set of $\beta$-indices which have been $\zeta$-used by stage $s$.  We will achieve $\lim G_s = \omega$.  Since $\C$ contains only infinite sets, every $\alpha$-indexed real is less than some $\beta$-indexed real, and therefore we can use $\beta$ as a garbage can to collect for those approximations $m_s(e)$ which turned out to be wrong.  We shall also have an auxiliary recursive function $r(s)$ which marks the boundary between the $\zeta$-indices which are following values in $\omega \union \{*\}$ and those whose value is $\diverge$ at stage $s$.

The construction is as follows:
\begin{description}

\item{\emph{Stage 0.}}

Set $G_0 = \emptyset$, $r(0) = 0$, $f_0(e) = \diverge$, and $\zeta_{e,0} = 0$ for all $e \geq 0$.

\item{\emph{Stage $s+1$.}}
Let
\begin{gather*}
A = \{x < s : \text{$m_{s+1}(x) \diverge$ and $m_{s}(x) \converge$}\}, \\
X = \{x < s : \text{$m_{s+1}(x) \converge$ and $m_{s}(x) \diverge$}\},
\end{gather*}
let $\{a_1, a_2, \dotsc a_k\}$ be the indices below or equal to $r(s)$ satisfying $f_s(e_i) \in A$, and let $\{x_1, x_2, \dotsc, x_d\}$ be the indices below or equal to $r(s)$ satisfying $f_s(e_i) \in X$.  We destroy all followers of $\{x_1, \dotsc, x_d\}$, and create new followers for $\{a_1, \dotsc, a_k\}$:
\begin{equation} \label{eqn1: kummer proof}
f_{s+1}(n) =
\begin{cases}
\pair{x_i, *} & \text{if $f_s(n) = \pair{x_i, \alpha}$ for some $1 \leq i \leq d$;}\\
\pair{a_i, \alpha} &\text{if $n = r(s) + i$ for some $1 \leq i \leq k$;} \\
\pair{s, \alpha} &\text{if $n = r(s) + k + 1$;} \\
\pair{\infty,*} &\text{if $n = r(s) + k + 2$;} \\
f_s(n) &\text{otherwise.}
\end{cases}
\end{equation}
The $\zeta$-index $r(s)+k+1$ is used to introduce a new $\alpha$-index, and the $\zeta$-index $r(s)+k+2$ is used to ensure that some new $\beta$-index is taken up at this stage.  Set
$r(s+1) = r(s) + k + 2$.

Next, assign new reals from $\B$ to the $\zeta$-indices that were destroyed in this stage.
\begin{itemize}
\item Let
\[
	y_1 = (\mu n)\: [\beta_{n,s} > \zeta_{x_1,s} \andd n \notin G_s]
\]
and inductively for $0 \leq i \leq d$,
\[
	y_{i+1} = (\mu n)\: [\beta_{n,s} > \max\{\zeta_{x_{i+1},s}, \beta_{y_i}\} \andd n \notin G_s].
\]
Choose the least $\beta$-index not yet assigned to a $\zeta$-index and call it $z$:
\begin{equation} \label{eqn: kummer proof z}
	z = \min\{n: \text{$n \notin \{y_1, y_2, \dotsc, y_d\}$ and $n \notin G_s$}\}.
\end{equation}
This choice of $z$ ensures that every member of $\B$ will have some index in $\zeta$.
\item Set
\begin{equation*} \label{eqn2: kummer proof}
	\zeta_{n,t} =
	\begin{cases}
	\beta_{y_i,t} & \text{if $f_{s+1}(n) = \pair{x_i,*}$ for some $1 \leq i \leq d$;} \\
	\beta_{z,t} & \text{if $n = r(s+1)$.}
	\end{cases}
\end{equation*}
for all $t>s$.

\item Set $G_{s+1} = G_s \union \{y_0, y_1, \dotsc, y_k, z\}$.
\end{itemize}

For the remaining $\zeta$-indices which have not been destroyed in this stage or some previous stage, continue following $\alpha$-indices:
\begin{equation} \label{eqn3: kummer proof}
	\zeta_{e,s+1} =
	\begin{cases}
	0 &\text{if $f_{s+1}(e) = \diverge$;} \\
	\alpha_{f_{s+1}(e),s+1} &\text{if $f_{s+1}(e) \notin \{\pair{n,*} :  n\in \omega\} \union \{\diverge\}$}.
	\end{cases}
\end{equation}
\end{description}
By induction on stages, \eqref{eqn1: kummer proof} and \eqref{eqn3: kummer proof} ensure that for all $s$ and $e \leq s$, there exists a unique $n$ such that
\[
	\zeta_{n,s+1} = \alpha_{\text{first projection of $f_{s+1}(e),s+1$}}.
\]
Since each sequence $\{\alpha_{\text{first projection of $f_{s+1}(e)$}}\}$ converges to a unique member in the range of $\alpha$ on the set of indices $e \in\dom m$, it follows that there is a unique $\zeta$-index for each real in $\A$.  Indeed for $e \notin \dom m$, the approximation for $m(e)$ may oscillate between convergence and divergence infinitely often, but we simply introduce a fresh $\zeta$-index for an unused member of $\B$ each time this happens and therefore $\alpha_e$ will not occupy a $\zeta$-index in the limit.   Furthermore \eqref{eqn1: kummer proof} and \eqref{eqn: kummer proof z} ensure that there is a unique $\zeta$-index for each real in $\B$.

Finally, $\zeta_e \in \A \union \B$ for all $e$.  If the index $e$ is destroyed at some stage in the construction, then some $\beta$-index $n$ is assigned at that stage and $\zeta_e = \beta_n$.  On the other hand if index $e$ is never destroyed, then $\zeta_e$ takes an $\alpha$-index, namely $\zeta_e = \alpha_{\text{first projection of $f(e)$}}$.
\noqed
\end{proof}
Hence \textsc{(i)}$\iff$\textsc{(ii)}$\iff$\textsc{(iii)}.
\end{proof}

\begin{cor} \label{cor: infinites are enumerable}
The following classes have left-r.e.\ numberings:
\begin{enumerate}[\scshape (i)]
\item the left-r.e.\ Martin-L\"{o}f random sets,
\item the left-r.e.\ Kurtz non-random sets,
\item the infinite left-r.e.\ sets, and
\item the infinite r.e.\ sets.
\end{enumerate}
\end{cor}
Proposition~\ref{prop: no r.e. numbering for infinite} below contrasts with Corollary~\ref{cor: infinites are enumerable}\nlb(\textsc{iv}).  This dichotomy does not surprise us too much as the recursive sets are also enumerable if viewed as r.e.\ characteristic functions, what is well-known to be impossible for recursive functions.  To see such an enumeration, we start with an enumeration of the binary p.r.\ functions, $f_0, f_1, f_2, \dotsc$.  We can uniformly interpret each $f_i$ as the recursive set whose characteristic function is the truncation of $f_i$ up to the highest number $n$ such that $f_i(x) \converge$ for all $x < n$, followed by the constant zero function.  Then the indices for total functions will yield the characteristic functions for the recursive sets, and the non-total functions will yield finite sets which are also recursive.
\begin{prop} \label{prop: no r.e. numbering for infinite}
There is no r.e.\ numbering of the infinite r.e.\ sets.
\end{prop}
\begin{proof}
Suppose that $A_0, A_1, A_2, \dotsc$ were an r.e.\ numbering of the infinite r.e.\ sets.  Search for an $a_0 \in A_0$, and let $b_0 = a_0+1$.  Next, search for an $a_1 \in A_1$ which is greater than $b_0$, and let $b_1 = a_1 + 1$.  Continuing the diagonalization, find $a_2 \in A_2$ which is greater than $b_1$ and let $b_2 = a_2 +  1$, and proceed similarly for $b_3, b_4, \dotsc$.  Now $\{b_0 < b_1 < b_2 < \dotsc\}$ is an infinite r.e.\ set which disagrees from the $n^\text{th}$ r.e.\ set at $a_n$.
\end{proof}

It remains to show that the hypothesis ``contains a shift-persistent element'' is necessary in Theorem~\ref{Sigma_3 iff numbering}.
\begin{thm} \label{thm: Sigma_3 not closed under shifts}
There exists a $\Sigma_3$-class of infinite left-r.e.\ reals which contains no shift-persistent element and has no left-r.e.\ numbering.
\end{thm}
\begin{proof}
Let $\alpha$ be a universal left-r.e.\ numbering and define the following $\alpha$-index set:
\begin{multline} \label{eqn:new3.8}
X = \{e: (\exists x)\: [x \notin \alpha_e \union \Omega] \\
\text{and } (\forall y<x)\: [y \in \alpha_e \iff y \notin \Omega] \\
\text{and } (\forall y>x)\: [ y \in \alpha_e \iff \text{$y$ is odd}]\}.
\end{multline}
By the third line, $X$ is infinite, and by the second line, $X$ contains no shift-persistent element.  Furthermore, \eqref{eqn:new3.8} is a $\Sigma_2$-formula with a $\0'$-recursive
predicate, hence $X \in \Sigma_3$.  If $X$ would have a left-r.e.\ numbering, then by the first line, $\complement{\Omega}$ would be the lexicographic supremum of all the approximations occurring to members of this left-r.e.\ numbering and $\complement{\Omega}$ would be a left-r.e.\ set, contradicting that $\Omega$ is nonrecursive.
\end{proof}
\noindent
Also along the lines of randomness, we note that the class of left-r.e.\ reals $X$ satisfying $X + \Omega \leq 1$ has a $\Pi_1$ index set (in any numbering), has no shift-persistent element, and has no left-r.e.\ numbering.  Indeed if this class had a left-r.e.\ numbering, then $\Omega$ would be recursive.

\begin{cor} \label{cor: ML are Sigma_3 - Pi_3}
The left-r.e.\ Martin-L\"{o}f non-random reals, computable non-random reals, and Schnorr non-random reals have no left-r.e.\ numberings.  Hence none of these classes has a $\Sigma_3$ index set in any universal left-r.e.\ numbering.
\end{cor}

\begin{proof}
These classes are $\Pi_3$-hard in any acceptable numbering by Corollary~\ref{cor: ML randoms are Sigma_3-hard}.  It follows from Theorem~\ref{Sigma_3 iff numbering} that none of these classes are effectively enumerable and hence cannot be $\Sigma_3$ in any universal left-r.e.\ numbering.
\end{proof}

\section{Weakly 1-generic sets}

We examine left-r.e.\ numberings for Kurtz random, bi-immune, bi-hyperimmune, and weakly 1-generic sets.  We introduced weakly 1-generic sets in Definition~\ref{def:1-generic}.

\begin{defn}\label{im}
An infinite set is \emph{immune} if it contains no infinite recursive subset. Even stronger, a set $A=\{a_0<a_1<\cdots\}$ is \emph{hyperimmune} if there exists no recursive function $f$ such that $f(n) > a_n$ for all $n$. It is \emph{bi-(hyper)immune} if both $A$ and the complement $\complement{A}$ are (hyper)immune.
\end{defn}

\begin{thm}
Let $A \subseteq \omega$ be a $\Sigma_3$-set, and let $\alpha$ be an acceptable universal left-r.e.\ numbering.  Then there exist a recursive function $g$ such that
\begin{align*}
x \in A & \implies \text{$\alpha_{g(x)}$ is co-finite;} \\
x \notin A & \implies \text{$\alpha_{g(x)}$ is weakly 1-generic.}
\end{align*}
\end{thm}
\begin{proof}
Let $W$ be an acceptable universal r.e.\ numbering. By the $\Sigma_3$-Representation Theorem, there exists a recursive function $f$ such that:
\begin{align*}
x \in A &\implies \text{$W_{f(x,n)}$ is infinite for some $n$;} \\
x \notin A &\implies \text{$W_{f(x,n)}$ is finite for all $n$.}
\end{align*}
The idea now is make $\alpha_{g(x)}$ a sequence of the form
\begin{equation*}
	\alpha_{g(x)} = \sigma_0{^\frown} 0{^\frown} \sigma_1{^\frown} 0 {^\frown} \sigma_2{^\frown} 0 \dotsb
\end{equation*}
such that $\sigma_0{^\frown} 0{^\frown} \sigma_1{^\frown} 0{^\frown} \dotsb {^\frown} \sigma_n$ is a member of $W_n$ whenever $W_n$ is dense and $W_{f(x,e)}$ is finite for all $e$.  If on the other hand $W_{f(x,e)}$ is infinite for some $e$, then some $\sigma_n$ will blow up to infinity and $\alpha_{g(x)}$ will be co-finite.

For every $n$, let
\[
\tau_{n,s} = \sigma_{0,s{^\frown}} 0{^\frown} \sigma_{1,s}{^\frown} 0{^\frown} \dotsb {^\frown}\sigma_{n,s}.
\]
At Stage~0, $\sigma_{e,0} = 0$ for all $e$, and $\alpha_{g(x),0} = \sigma_{0,0}{^\frown} 0{^\frown} \sigma_{1,0}{^\frown} 0 {^\frown} \sigma_{2,0}{^\frown} 0{^\frown}\dotsb$.
At Stage~$s+1$, let $e$ be the least index, if one exists, such that either:
\begin{enumerate}
\item $\size{W_{f(e,x),s+1}} > \size{W_{f(e,x),s}}$, or
\item some member of $W_{e,s+1}$ extends $\tau_e{^\frown} 1$, and no member of $W_{e,s+1}$ is a prefix of $\tau_e$.
\end{enumerate}
If condition~1 is satisfied, let $\sigma_{e,s+1}$ = $\sigma_{e,s}{^\frown} 1$ so that $\sigma_e$ becomes longer.  Otherwise, let $\sigma_{e,s+1}$ be an extension of $\sigma_{e,s}{^\frown} 1$ such that $\tau_{e,s+1} \in W_{e,s+1}$.  This latter case aims to make $\alpha_{g(x)}$ weakly~1-generic.  In either case, $\sigma_{j,s+1} = \sigma_{j,s}$ for $j \neq e$.  If no such $e$ exists, skip to Stage~$s+2$.

The sequence $\{\alpha_{g(x),s}\}$ is indeed a left-r.e.\ approximation.  In each stage~$s$ where some action takes place in the construction, the 0 following $\sigma_{e,s}$ is changed to a 1 before this string is extended.

We claim that if $W_{f(x,n)}$ is finite for all $n$, then $\alpha_{g(x)}$ will be weakly 1-generic.  By some stage~$s$, $W_{f(0,x),s}$ must stop expanding.  Whether or not $W_0$ is dense, $\sigma_0$ will change at most one time after stage $s$, and therefore $\tau_0$ settles by some stage~$t_0$.   If $W_0$ is dense, then $\tau_0$ will contain a member of $W_0$.  Similarly $W_{f(1,x),s}$ must stop expanding at some point after stage~$t_0$, $\tau_1$ will eventually contain a prefix of $W_1$ if $W_1$ is dense, and $\tau_{1,t}$ settles by some stage $t_1$.  Continuing by induction, we see that $\alpha_{g(x)}$ is weakly 1-generic.

If $W_{f(x,n)}$ is infinite for some least $n$, then the argument in the previous paragraph shows that $\tau_{n-1,t}$ eventually settles, and then infinitely often $\sigma_{n,s+1} = \sigma_{n,s}{^\frown}1$ and so $\alpha_{g(x)}$ is co-finite.
\end{proof}
\noindent
Since every weakly 1-generic set is both hyperimmune \cite[Proposition~1.8.48]{N09} and Kurtz random (Theorem~\ref{thm: 1-generic implies Kurtz random}) we have the following:
\begin{cor}\label{bi-hard}
In any acceptable universal left-r.e.\ numbering, the index sets for the following classes are $\Pi_3$-hard:
\begin{compactenum}[\scshape (i)]
\item the left-r.e.\ immune sets,
\item the left-r.e.\ hyperimmune sets,
\item the left-r.e.\ bi-immune sets,
\item the left-r.e.\ bi-hyperimmune sets,
\item the left-r.e.\ weakly 1-generic sets,
\item the left-r.e.\ Kurtz random sets.
\end{compactenum}
\end{cor}

\noindent
From Theorem~\ref{Sigma_3 iff numbering} we also have the following result.
\begin{cor}\label{not bi-hard}
In any universal left-r.e.\ numbering, the classes listed in Corollary~\ref{bi-hard} have $\Pi_3 - \Sigma_3$ index sets.  Moreover there exists a left-r.e.\ numbering for each of the corresponding complementary classes.
\end{cor}

It is known that every Kurtz random is bi-immune \citep{Kur81}, but the reverse inclusion does not hold \citep{BST10}.  We can also separate the left-r.e.\ versions of these notions.
\begin{prop}\label{bhi not kr}
There exists a left-r.e.\ bi-hyperimmune set which is not Kurtz random.
\end{prop}
\begin{proof}
Let $A$ be any bi-hyperimmune left-r.e.\ set.  Then $A \join A$ is bi-hyperimmune but not Kurtz random since a recursive martingale can win on every second bit.
\end{proof}
\noindent
The reverse of Proposition \ref{bhi not kr} holds as well: Chaitin's $\Omega$ is an example of a left-r.e.\ Martin-L\"{o}f random (in particular, Kurtz random) which, by Lemma~\ref{lots of zeros implies nonrandom}, is not hyperimmune (in particular, not bi-hyperimmune).

\section{Classes of higher complexity}
We now investigate the complex randomness notions of Schnorr randomness and computable randomness.  As we shall see, neither of these left-r.e.\ classes, nor their complements, have left-r.e.\ numberings.  A set $A$ is called \emph{high} if $A' \geq_\T \0''$.  A theorem of Nies, Stephan, and Terwijn \citep{NST05} shows the existence of left-r.e.\ Schnorr randoms which are not Martin-L\"{o}f random:
\begin{thm}[Nies, Stephan, Terwijn \citep{NST05}] \label{thm: high separation}
The following statements are equivalent for any set $A$:
\begin{compactenum}[\scshape (i)]
\item $A$ is high.
\item There is a set $B \equiv_\T A$ which is computably random but not Martin-L\"{o}f random.
\item There is a set $C \equiv_\T A$ which is Schnorr random but not computably random.
\end{compactenum}
In the case that $A$ is left-r.e.\ and high, the sets $B$ and $C$ can
be chosen as left-r.e. sets as well.
\end{thm}
\noindent
Furthermore, Downey and Griffith \cite{DG04, DH10} proved that every left-r.e.\ Schnorr random real is high.  Therefore
\begin{fact} \label{cor: Schnorr iff high}
A left-r.e.\ set $X$ is high $\iff$ $X$ Turing equivalent to a left-r.e.\ Schnorr random set $\iff$ $X$ is Turing equivalent to a left-r.e.\ computably random set.
 \end{fact}

In his PhD thesis \citep{Sch82}, Schwarz characterized the complexity of the high r.e.\ degrees:
\begin{thm}[\citet{Sch82}, \citep{Soa87}]  \label{thm: schwarz's theorem}
In any acceptable universal r.e.\ numbering $W_0, W_1, W_2, \dotsc$, $\{e : \text{$W_e$ is high}\}$
is $\Sigma_5$-complete.
\end{thm}
\noindent
Using this Schwarz's theorem, we obtain the following enumeration result.
\begin{thm} \label{thm: no numbering for high}
Let $\C$ be a class of left-r.e.\ reals such that:
\begin{enumerate}[\scshape (i)]
\item Every member of $\C$ is high, and
\item every high r.e.\ set is Turing equivalent to some member of $\C$.
\end{enumerate}
Then for any universal left-r.e.\ numbering $\alpha$, $\{e : \alpha_e \in \C\}$ is not a $\Sigma_4$-set and hence is neither enumerable nor co-enumerable.
\end{thm}
\begin{proof}
Let $\C$ be a class satisfying the hypothesis of the theorem, let $W$ be an acceptable universal r.e.\ numbering, let $\Phi$ denote a Turing functional, and suppose that
\[
	\alpha_i \in \C \iff (\exists n_1)\: (\forall n_2)\: (\exists n_3)\: (\forall n_4)\: [P(i,n_1,n_2,n_3,n_4)].
\]
for some recursive predicate $P$.

For convenience assume that whenever a computation $\Phi^{W_{e,t}}_j$ is injured, it is undefined for at least one stage; then
\begin{gather*}
\text{$W_e$ is high} \iff (\exists i,j) \left[\alpha_i \in \C \andd \alpha_i = \Phi^{W_e}_j \right] \\
\begin{split}
\iff (\exists i,j,n_1)\: &(\forall x,n_2)\: (\exists t)\: (\exists n_3)\: (\forall u > t)\: (\forall n_4) \\ & \left[P(i,n_1,n_2,n_3,n_4) \andd \alpha_{i,u}(x) =  \Phi_{j,u}^{W_{e,u}}(x) \right].
\end{split}
\end{gather*}
Thus $\{e : \text{$W_e$ is high}\}$ is a $\Sigma_4$-set, contrary to Theorem~\ref{thm: schwarz's theorem}.
\end{proof}

\begin{cor}\label{Neither}
Neither the Schnorr random sets nor the computably random sets reals are $\Sigma_4$ in any universal left-r.e.\ numbering.  Hence neither class nor its complement has a left-r.e.\ numbering.
\end{cor}
\begin{proof}
By Fact~\ref{cor: Schnorr iff high}, the left-r.e.\ Schnorr random sets and the left-r.e.\ computably random sets satisfy the hypothesis of the Theorem~\ref{thm: no numbering for high}.  Apply Theorem~\ref{Sigma_3 iff numbering}.
\end{proof}

It remains to characterize the hardness of computable random sets and Schnorr random sets in an acceptable universal left-r.e.\ numbering.  For the remainder of this paper, we fix an acceptable universal left-r.e.\ numbering $\alpha$ and an acceptable universal r.e.\ numbering $W$.  The \emph{principal function} of a set $A = \{a_0 < a_1 < a_2 < \dotsc\}$ is given by $n \mapsto a_n$; we write $p_{A}(n)=a_{n}$. We will be particularly interested in the principal functions of co-r.e.\ sets, so we use the abbreviation $p_{\complement{e}}$ for $p_{\complement{W}_e}$.  We say that a function $f:\omega\to\omega$ is \emph{dominating} if it dominates all recursive functions.

\begin{thm}\label{using NST}
There is a Turing reduction procedure $\Phi$ and a recursive function $g$ such that for all $e$,
\begin{enumerate}[\scshape (i)]
\item $\Phi^{p_{\complement{e}}}$ is a left-r.e. real,
\item $\alpha_{g(e)}=\Phi^{p_{\complement{e}}}$, and
\item $\Phi^{f}$ is computably random if $f$ is dominating.
\end{enumerate}
\end{thm}
\begin{proof}
This fact follows from the proof of Nies, Stephan, Terwijn \cite[Theorem 4.2, (I) implies (II), r.e.\ case]{NST05} which appears as Theorem~\ref{thm: high separation} in this paper.
\end{proof}

\noindent
A set $A$ is \emph{low} if $A'\le_{\T}\0'$, and a function is \emph{low} if it is computable from a low set.  A function $f$ is \emph{diagonally non-recursive (DNR)} if for some numbering $\phe$ and every $e$, the value $\phe_e(e)$, if defined, differs from $f(e)$.

\begin{lemma}\label{observe}
A low left-r.e.\ set cannot compute a Schnorr random.
\end{lemma}
\begin{proof}
Suppose such a real $A$ computes a Schnorr random set $X$.  Since $X$ is not high, $X$ must also be Martin-L\"{o}f random (by Theorem~\ref{thm: high separation}).  Ku\v{c}era showed that every Martin-L\"{o}f random set computes a DNR function \cite{Kuc85}, \cite[Theorem 6]{KMS06}, so $A$ computes a DNR function. Moreover $A$ has r.e.\ Turing degree because it is truth-table equivalent to the r.e.\ set $\{\sigma: \sigma^{\frown} 0^\omega \leq A\}$.  An r.e.\ set computes a DNR-function if and only if the set is Turing complete \cite{Ars81}\cite{JLSS89}\cite[Corollary 9]{KMS06}\cite{KMS11}, hence $A \equiv_\T \0'$.  This contradicts the fact that $A$ is low.
\end{proof}

\noindent
An r.e.\ set $A$ is \emph{maximal} if for each r.e.\ set $W$ with $A\subseteq W$, either $\omega\setminus W$ or $W\setminus A$ is finite. Friedberg \cite{Fri58} proved that maximal sets exist.
\begin{thm}
For every $A \in \Pi_4$ and every acceptable left-r.e.\ numbering~$\alpha$, there exists a recursive function $f$ such that for all $e$,
\begin{align}
e \in A &\implies \text{$\alpha_{f(e)}$ is computably random;}  \label{pi four eq}\\
e \notin A &\implies \text{$\alpha_{f(e)}$ is not Schnorr random.}
\end{align}
\end{thm}
\begin{proof}
Let us fix a $\Pi_{4}$-complete set $A$; By \cite[XII.\ Exercise 4.26]{Soa87}, there is a recursive function $h$ such that
\[
	e\in A\quad\iff\quad W_{h(e)}\text{ is maximal}\quad\iff\quad W_{h(e)}\text{ is not low}.
\]
Martin and Tennenbaum showed that the principal function of the complement of a maximal set dominates all recursive functions  \cite[XI.\ Proposition~1.2]{Soa87}.  Using this result together with the function $g$ and operator $\Phi$ given by Theorem~\ref{using NST},
\[
	W_{h(e)}\text{ is maximal}\quad\implies\quad \text{$p_{\complement{h(e)}}$ is dominating}
\]
\[
	\quad\implies\quad \text{$\alpha_{g[h(e)]}=\Phi^{p_{\complement{h(e)}}}$ is computably random},
\]
and by Lemma~\ref{observe} with Theorem~\ref{using NST}\nlb\textsc{(i)},
\[
	\text{$W_{h(e)}$ is not maximal}\quad\implies\quad \text{$p_{\complement{{h(e)}}}$ is low}
\]
\[
	\quad\implies\quad \alpha_{g[h(e)]}=\text{$\Phi^{p_{\complement{h(e)}}}$ is not Schnorr random}.
\]
The function $f = g\circ h$ witnesses the conclusion of this theorem.
\end{proof}
\noindent
Note that if we replaced ``computably random'' with ``Martin-L\"{o}f random'' in \eqref{pi four eq}, we would obtain a characterization of $\Sigma_3$ sets rather than $\Pi_4$ sets (care of Theorem~\ref{thm: ML Sigma_3-hard}).  Since every computable random is Schnorr random (Theorem~\ref{ML implies computable implies Schnorr}), we obtained an optimal hardness result:
\begin{cor} \label{complete}
In any acceptable universal left-r.e.\ numbering, both the indices of the Schnorr random sets and the indices of the computably random sets are $\Pi_{4}$-complete.
\end{cor}

We summarize our main results in Table~\ref{vege-table}.  A theorem in a forthcoming paper \citep{ST10b} states that every $\0'$-recursive 1-generic set has a co-r.e.\ indifferent set which is retraceable by a recursive function.  It follows that for each the families of randoms listed in Table~\ref{vege-table}, there exists a universal left-r.e.\ numbering which makes the set of the indices for that class 1-generic.  Therefore we cannot obtain any arithmetic hardness results for index sets in the general case of universal left-r.e.\ numberings.

\begin{table}[h]
\begin{center}
\begin{tabular}{|c|c|c|c|c|}
\hline
\textsf{Left-r.e.\ family}  	& \textsf{Complexity} 							& \textsf{Hardness*}				\\ \hline
Martin-L\"{o}f randoms   	&  $\Sigma_3 -\Pi_3$ [\ref{cor: ML are Sigma_3 - Pi_3}] 	& $\Sigma_3$-hard [\ref{cor: ML randoms are Sigma_3-hard}]  \\ \hline
computable randoms       &   $\Pi_4 - \Sigma_4$ [\ref{Neither}] 				& $\Pi_4$-hard [\ref{complete}] 		\\ \hline
Schnorr randoms 		& $\Pi_4 - \Sigma_4$ [\ref{Neither}] 					& $\Pi_4$-hard [\ref{complete}] 		\\ \hline
Kurtz randoms 			& $\Pi_3 - \Sigma_3$ [\ref{not bi-hard}] 	& $\Pi_3$-hard [\ref{bi-hard}] 	\\ \hline
bi-immune sets 		& $\Pi_3 - \Sigma_3$ [\ref{not bi-hard}] 				& $\Pi_3$-hard [\ref{bi-hard}] 		\\ \hline
\end{tabular}
\caption{Complexities listed hold for any universal left-r.e.\ numbering. *Hardness results are for \emph{acceptable} universal left-r.e.\ numberings.}
\label{vege-table}
\end{center}
\end{table}

We can separate most of the adjacent left-r.e\ classes in Table~\ref{vege-table} simply by observing differences in arithmetic complexity (and using the well-known result Theorem~\ref{ML implies computable implies Schnorr}).  The remaining separations follow from Theorem~\ref{thm: high separation} and Proposition~\ref{bhi not kr}.  All of these separations were previously known, with the possible exception of a left-r.e.\ Kurtz random which is not bi-immune.

Among the families Table~\ref{vege-table}, only the Martin-L\"{o}f randoms have a left-r.e.\ numbering, and among the complementary families only the Kurtz non-randoms and non-bi-immune sets do (by Theorem~\ref{Sigma_3 iff numbering}).

\section{Expanding the vocabulary}
In Section~\ref{sec: acceptable}, we identified left-r.e.\ sets as limit-recursive sets with recursive approximations of a special form.  However there are other easy to describe limit-recursive sets which are Martin-L\"{o}f random but not left-r.e.  For example, $\{x : 2x \in \Omega\}$ is Martin-L\"{o}f random and low by van Lambalgen's Theorem \cite[Corollary~3.4.11]{N09} but not left-r.e.\ (as left-r.e.\ Martin-L\"{o}f random sets are weak truth-table complete \cite[Corollary~3.2.31]{N09}).  See \cite[Proposition~13]{JST11} for an elementary explanation why $\{x : 2x \in \Omega \}$ and $\{x : 2x+1 \in \Omega \}$ cannot both have left-r.e.\ approximations.
\begin{ques}
If $A = a_0 < a_1 < a_2 <\dotsc$ is an infinite r.e.\ (co-r.e.) set, and $\Omega$ is a left-r.e.\ Martin-L\"{o}f random set, is the set
\begin{equation}\label{r.e. slice}
\Omega(a_0)\Omega(a_1)\Omega(a_2)\dotsc
\end{equation}
Martin-L\"{o}f random?  If not, which classes of sequences  of the form \eqref{r.e. slice} have numberings?
\end{ques}

\bibliographystyle{acmtrans}
\bibliography{nonrandom_enumeration}

\end{document}